\title{Exponential Vorticity Hessian Growth in Capillary Liquid Drop in Two Dimensions}
\author{Zhongtian Hu\thanks{
                  Department of Mathematics, Princeton University, Princeton, NJ 08544, USA; email: \url{zh1077@princeton.edu}}
        }
\documentclass[11pt]{article}
\usepackage[top=1in, bottom=1.in, left=1.in, right=1.in]{geometry}

 \usepackage[bookmarksnumbered, bookmarksopen, colorlinks, citecolor=red, linkcolor=black]{hyperref}
 \usepackage{amsfonts}
\usepackage{graphicx}
 \usepackage{amsmath, amssymb}
 \usepackage{amsthm}
 \usepackage{comment}
 \usepackage{mathtools}
 \usepackage{calc}
 \usepackage{cancel}
 \usepackage{accents}

\usepackage{accents}

\numberwithin{equation}{section}
\newtheorem{thm}{Theorem}[section]
\newtheorem{lem}[thm]{Lemma}
\newtheorem{prop}[thm]{Proposition}
\newtheorem{rmk}[thm]{Remark}

\newtheorem{defn}[thm]{Definition}

\usepackage{xcolor}
\definecolor{purple}{rgb}{0.5, 0, 1}
\definecolor{orange}{rgb}{1,.5,0}

\newcommand{\eps}{\varepsilon}

\newcommand{\R}{\mathbb{R}}

\newcommand{\calD}{\mathcal{D}}
\newcommand{\calN}{\mathcal{N}}
\newcommand{\calH}{\mathcal{H}}

\newcommand{\p}{\partial}

\begin{document}
\newpage
\maketitle
%%%%%%%%%%%%%%%%%%%%%%%%%%%
% abstract, keywords, and Subject classification are optional.
%%%%%%%%%%%%%%%%%%%%%%%%%%%
\begin{abstract}
In this work, we concern ourselves with the evolution of a droplet of an ideal fluid in two dimensions, which has nontrivial bulk vorticity and is only subject to the effects of surface tension. We construct initial data with initial domain being a disk and initial velocity being arbitrarily small, such that the vorticity Hessian grows exponentially infinitely in time. 
\end{abstract}

%%%%%%%%%%%%%%%%%%%%%%
% % Here is the start of the Text
%%%%%%%%%%%%%%%%%%%%%%
\section{Introduction}
We consider the motion of an incompressible fluid in two dimensions with a free boundary separating the moving fluid region $\calD_t$ and a vacuum exterior. Inside the fluid bulk $\calD_t$, the fluid velocity $u(t,x)$ and the pressure $p(t,x)$ satisfy the incompressible Euler equations:
\begin{equation}
    \label{eq:fbeulermodel}
    \begin{cases}
        \p_t u + u\cdot \nabla u + \nabla p = 0,& \text{ in } \calD_t,\\
        \nabla\cdot u = 0,& \text{ in } \calD_t.
    \end{cases}
\end{equation}
In this work, we focus on the scenario of a capillary liquid drop: that is, the fluid region $\calD_t$ is a bounded, simply connected domain with boundary $\Gamma_t := \p\calD_t$, whose motion is governed by the kinematic boundary condition:
\begin{equation}
    \label{eq:kbc}
    V = u\cdot \calN\quad \text{on $\Gamma_t$}.
\end{equation}
Here, $\calN$ denotes the outward unit normal with respect to $\Gamma_t$. Moreover, the fluid motion is under the effect of capillarity i.e. the presence of surface tension on the free boundary $\Gamma_t$. Such effect is characterized by the Young-Laplace equation:
\begin{equation}
    \label{bdry_b}
    p = \sigma \calH \quad \text{on $\Gamma_t$},
\end{equation}
where $\sigma > 0$ is the surface tension coefficient and $\calH$ is the mean curvature of $\Gamma_t$.

The free-boundary Euler equations with surface tension \eqref{eq:fbeulermodel}--\eqref{bdry_b} has a rich history of research. A classical scenario is when the free interface $\Gamma_t$ is a graph of the horizontal variable $x_1$, in which case the system \eqref{eq:fbeulermodel}--\eqref{bdry_b} classically describes the capillary water waves on the ocean surface. The mathematical study of this scenario has blossomed over the past few decades, and we briefly mention a few here. A quite general local well-posedness theory for \eqref{eq:fbeulermodel}--\eqref{bdry_b} was established in \cite{coutand2007well,SZ1,SZ2,SZ3} in high regularity spaces. When one considers the special case where the flow is irrotational, local regularity theory was proved in all subcritical Sobolev spaces, see \cite{ABZ11}. Moreover, thanks to a dispersive nature as well as a favorable nonlinear structure in the irrotational case, one is able to establish long-time/global existence for solutions initiated by sufficiently small data, see \cite{berti2018almost,berti2021quadratic,DIPP17,ifrim2017lifespan,ionescu2018global}. On the contrary, it is known from \cite{castro2012finite,castro2012splash,castro2013finite,coutand2014finite} that wave breaking and splash singularities can occur for \eqref{eq:fbeulermodel}--\eqref{bdry_b} given large initial data that are close to the turn-over situation.

Arguably, another interesting scenario is when the fluid domain is a spherical droplet, whose boundary is only subject to the surface tension effects. While the study of motions of capillary droplets in the absence of gravity has attracted much attention from physicists (see e.g. \cite{tsamopoulos1983nonlinear}), the mathematical exploration of this subject remains relatively few comparing with the classical capillary-gravity water waves with a graph-type interface. On one hand, when the fluid flow is irrotational, a local well-posedness theorem was proved in \cite{beyer1998cauchy} via a Nash-Moser approach. Very recently, in a series of works \cite{shao2022longtime,shao2023differential,shao2023toolbox,shao2023cauchy}, the author developed a geometrically inspired paradifferential framework to establish a sharper local well-posedness theory. Beyond local well-posedness, rotational traveling waves on the boundary of a droplet are found in both two and three dimensions via bifurcation techniques, see \cite{moon2024global,baldi2024liquid,baldi2025bifurcation,la2025two}. On the other hand, when one considers droplets with general bulk velocity, much less is known perhaps except for the general well-posedness theory stated in \cite{coutand2007well,SZ1,SZ2,SZ3} to the best of the author's knowledge. Therefore, it is interesting to study quantitative behaviors for capillary droplets when the fluid flow is actually rotational. 

In this work, we address the \textit{instability} issue regarding the system \eqref{eq:fbeulermodel}--\eqref{bdry_b}. Motivated by the construction given by Zlato{\v s} in \cite{zlatovs2015exponential}, we build initial data with the initial domain being a disk, and the initial velocity being smooth and arbitrarily small, such that $\|u(t,\cdot)\|_{W^{3,\infty}}$ grows infinitely in time in a $\limsup$ sense. We now state our main result.

 \begin{thm}\label{thm:smallscale}
Consider the system \eqref{eq:fbeulermodel}--\eqref{bdry_b} with $\sigma=1$ and initial domain $\calD_0 = B_2$. There exists a smooth vector field $v_0 \in C^\infty(\calD_0)$ and a threshold $\eps_0 > 0$ such that for any $\eps\in (0,\eps_0)$, the following holds for the regular solution to \eqref{eq:fbeulermodel}-\eqref{bdry_b} with initial velocity $u_0 := \eps v_0$: there exists $T_1(\eps) > 0$ such that either $(u,\calD_t)$ ceases to be a regular solution for some time $T \le T_1$, or for any $T > T_1$,
\begin{equation}
\label{est:main}
\sup_{t \le T}\|\nabla^2 \omega(t,\cdot)\|_{L^\infty(\calD_t)}\ge \eps e^{\eps T},
\end{equation}
where $\omega:=\partial_1 u_2-\partial_2 u_1$ is the bulk vorticity.
\end{thm}

\begin{rmk}
    \begin{enumerate}
        \item Throughout this work, a regular solution to \eqref{eq:fbeulermodel}--\eqref{bdry_b} means a solution with spatial regularity $H^s$ with $s$ sufficiently large. The initial data that we construct will be in $H^\infty$ class, which can be evolved as an $H^s$ solution in view of \cite{coutand2007well,SZ1}.
        \item Theorem \ref{thm:smallscale} can be readily extended to the case for arbitrary $\sigma > 0$ in view of the following scaling symmetry: given $(u(t,x),\calD_t)$ solving \eqref{eq:fbeulermodel}--\eqref{bdry_b} with $\sigma = 1$,\\ $(\sigma^{1/2}u(\sigma^{1/2}t,x), \calD_{\sigma^{1/2}t})$ solves \eqref{eq:fbeulermodel}--\eqref{bdry_b} for arbitrary $\sigma > 0$.
        \item The solutions considered in Theorem \ref{thm:smallscale} may lose regularity in finite time in very specific ways in view of a Beale-Kato-Majda type continuation criterion recently developed in \cite{hao2025classification}.
    \end{enumerate}
\end{rmk}
    We compare Theorem \ref{thm:smallscale} to the main result of \cite{hu2024small}, which demonstrates double-exponential growth in vorticity gradient for 2D periodic capillary-gravity water waves. In both works, we regard the dynamics of the free boundary as perturbation by using a conserved quantity, and devise growth mechanisms which are sufficiently separated from the free boundary. However, there are two key differences. First, the control of the free interface in the case of liquid droplet is more involved: a geometric argument involving a use of isoperimetric inequality is deployed to ensure that the effects by the free interface are indeed perturbative. Second, the growth mechanism presented in this work is different. Unlike \cite{hu2024small} which crucially utilizes the fixed part of the boundary, we have to rely on a growth mechanism which is independent from any fixed boundary in the droplet case. Yet, we still can obtain growth on the level of vorticity Hessian by a construction inspired by \cite{zlatovs2015exponential}.
\subsection{Strategy of the Proof}
The crux of the proof is contingent upon the construction of strong hyperbolic flows, which has been a successful candidate to create small scales in 2D Euler equations in fixed domains. In the seminal work \cite{kiselev2014small}, the authors exploits such hyperbolic flow structure at the boundary of a fixed disk to prove a double-exponential growth rate for vorticity gradient. However, the construction in \cite{kiselev2014small} cannot be readily adapted to the droplet case. This is because the construction in \cite{kiselev2014small} crucially utilizes a hyperbolic point which is exactly on the fixed boundary. This scenario might not be true in the free boundary case, as the free interface is subject to the kinematic boundary condition instead of the no-flux boundary condition. Therefore, inspired by a construction given by Zlato{\v{s}} in \cite{zlatovs2015exponential}, we aim to construct strong hyperbolic flows strictly in the fluid bulk (close to the origin, in our case) and prove the desired growth result without help from physical boundaries. Note that such hyperbolic flows behoove us to work with a special symmetry which is reminiscent of the odd-odd symmetry in the classical 2D Euler case. We give a precise definition of this symmetry and prove the conservation of which under the evolution of \eqref{eq:fbeulermodel}--\eqref{bdry_b} in Section \ref{sec:symmetry}.

Unsurprisingly, a critical issue preventing us from realizing a construction in the spirit of \cite{zlatovs2015exponential} is the free boundary. Due to nonlocality and the kinetic boundary condition conformed by the free boundary, one expects that the flow might deviate significantly from a hyperbolic flow when the free interface gets sufficiently close to the origin, where the growth mechanism undergoes. A crucial observation is that such adverse scenario will never happen given sufficiently small initial kinetic energy. This observation relies on the key conserved quantity: $K(t) + \sigma L(t)$, where $K(t)$ is the kinetic energy and $L(t)$ is the length of the free interface. In our case, where the initial domain is a circular disk, we develop a geometric lemma (cf. Lemma \ref{prop:bdryctrl}) which shows that the free boundary will always be outside a smaller disk $B$ with a fixed radius. We establish the aforementioned geometric lemma in Section \ref{subsect:bdryctrl}.

To proceed, we follow the framework developed in \cite{hu2024small} by the author and collaborators: we devise an approximated Biot-Savart law (cf. Proposition \ref{prop:key}) which gives a pointwise description of the velocity field $u$ nearby the origin. The idea is to view the full velocity $u$ as the superposition of a \textit{main part} $U$ and an \textit{error term} $e$, where $U$ is the exact incompressible Euler flow in the \textit{fixed} domain $B$, and $e$ encodes the boundary effects. We will show that $e$ is perturbative and therefore $u$ verifies a pointwise characterization that is similar to the ones established in \cite[Lemma 3.1]{kiselev2014small} and \cite[Lemma 2.1]{zlatovs2015exponential}. This procedure will be performed in Section \ref{sec:bs}. In Section \ref{sec:main}, we use the key Proposition \ref{prop:key} to prove Theorem \ref{thm:smallscale} by following the ideas in \cite{zlatovs2015exponential}.

 \begin{figure}[thbp]
 \begin{center}
 \includegraphics[scale=0.8]{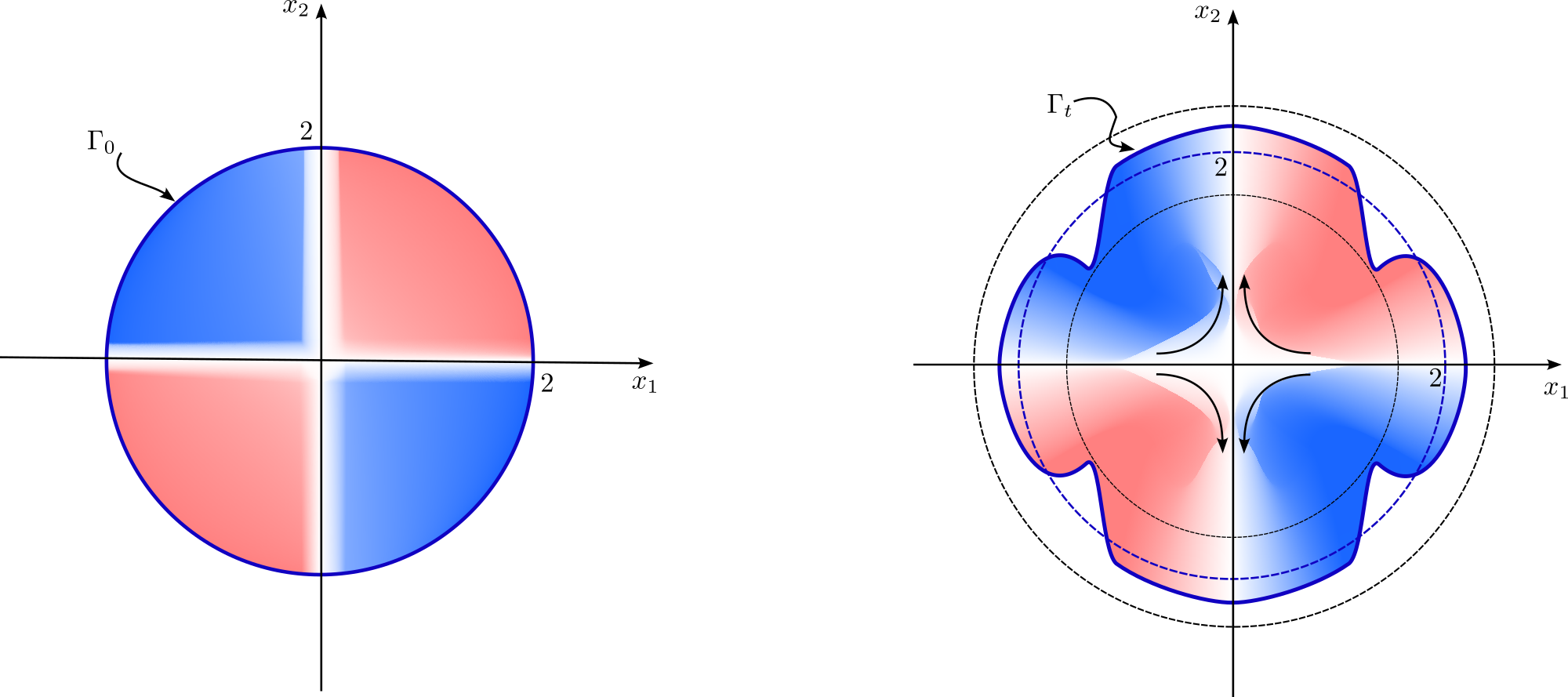}
 \end{center}
 \caption{\label{fig1} A picture of initial data chosen in Theorem \ref{thm:smallscale} and its possible evolution. The vorticity is positive in the red region and negative in the blue region. With the chosen data, $\Gamma_t$ is always constrained in a thin annulus about $\p B_2$. See Section \ref{subsect:bdryctrl} for a detailed discussion.}
 \end{figure}

We conclude the introduction by mentioning a few frequently used notations and conventions throughout this work. We let $B_r(p)$ be the open disk centered at $p$ with radius $r$, and write $B_r := B_r(0)$. We denote by $C$ universal constants whose values may change from line to line. Any constants with subscripts, such as $C_i$ or $c_i$, will stay fixed once they are chosen. Lastly, we use summation convention throughout this paper: repeated indices will be summed over.

 \subsection*{Acknowledgements}
 
The author is partially supported by the NSF-DMS grants 2306726 and 2406293. The author thanks Siming He for stimulating discussions.

\section{Symmetry and Conserved Quantities}\label{sec:symmetry}
In this section, we establish a crucial symmetry property and remark on several conserved quantities concerning the free boundary Euler equations with surface tension.  We start with showing a class of symmetries that is conserved by the evolution of the 2D free boundary Euler equations when the fluid domain is under certain symmetry assumption. Such symmetry is reminiscent of the classical ``odd-odd" symmetry which induces hyperbolic flows in the study of 2D Euler equations in fixed domains. 

We first make precise of a terminology regarding a symmetry which the fluid domain will satisfy.

\begin{defn}
    We say that a domain $D$ is 2-fold reflective symmetric if $D = \bar{D} = \tilde{D}$, where
    $$
    \bar{D} := \{(x_1,-x_2)\;|\; (x_1,x_2) \in D\},\quad \tilde{D} := \{(-x_1,x_2)\;|\; (x_1,x_2) \in D\}.
    $$
\end{defn}
With the definition given above, we state an important symmetry property which is conserved by the free-boundary Euler equation.
\begin{lem}\label{lem: symmetry}
Let $(u_0, \calD_0)$ be initial data of \eqref{eq:fbeulermodel}--\eqref{bdry_b}, such that $\calD_0$ is 2-fold reflective symmetric and $u_0 = (u_{01}, u_{02})$ satisfies 
\begin{equation}\label{symmetry t=0}
\begin{split}
-u_{01} (-x_1, x_2) = u_{01} (x_1, x_2) = u_{01}(x_1,-x_2),\\
u_{02} (-x_1, x_2) =u_{02} (x_1, x_2) = -u_{02}(x_1,-x_2).
\end{split}
\end{equation}
Then for all times during the lifespan of a solution, the solution $(u, \calD_t)$ satisfies the same symmetry, i.e.
\begin{equation}\label{symmetry u}
\begin{split}
-u_1 (t, -x_1, x_2) = u_1 (t, x_1, x_2) = u_1(t,x_1,-x_2),\\
u_2 (t, -x_1, x_2) = u_2 (t, x_1, x_2) = -u_2(t,x_1,-x_2),
\end{split}
\end{equation}
and the moving fluid domain $\calD_t$ is 2-fold reflective symmetric.
\end{lem}
\begin{proof}
    To prove the lemma, we first observe the following symmetry with respect to $x_2$ axis: suppose initial data $(u_0, \calD_0)$ satisfies
    $$
    -u_{01} (-x_1, x_2) = u_{01} (x_1, x_2),\quad u_{02} (-x_1, x_2) =u_{02} (x_1, x_2),
    $$
    and $\calD_0 = \tilde{\calD}_0$. Let $(u,\calD_t)$ be the corresponding solution. Then a straightforward computation (see \cite[Lemma 2.1]{hu2024small}) shows that $(\tilde{u},\tilde{\calD}_t)$ is also a solution to \eqref{eq:fbeulermodel}--\eqref{bdry_b}, where
    $$
    \tilde{u}(t,x_1,x_2) = (-u_1(t,-x_1,x_2), u_2(t,-x_1,x_2)),\quad \tilde{\calD}_t = \{(-x_1,x_2)\;|\; (x_1,x_2) \in \calD_t\}.
    $$
    By a computation which is similar to that showing the aforementioned symmetry property, one is also able to prove the following symmetry with respect to $x_1$ axis: suppose initial data $(u_0, \calD_0)$ satisfies
    $$
    u_{01} (x_1, -x_2) = u_{01} (x_1, x_2),\quad u_{02} (x_1, -x_2) = -u_{02} (x_1, x_2),
    $$
    and $\calD_0 = \overline{\calD}_0$. Then $(\bar{u},\overline{\calD_t})$ is also a solution, where
    $$
    \bar{u}(t,x_1,x_2) = (u_1(t,x_1,-x_2), -u_2(t,x_1,-x_2)),\quad \overline{\calD}_t = \{(x_1,-x_2)\;|\; (x_1,x_2) \in \calD_t\}.
    $$
    The lemma is thus proved by combining the two symmetries together.
\end{proof}

\begin{rmk}\label{rmk:vortsym}
A straightforward consequence of \eqref{symmetry u} is that the vorticity $\omega(t,x) = \nabla^\perp \cdot u(t,x)$ is odd in both $x_1$ and $x_2$, as long as the solution remains regular.
\end{rmk}

Next, we recall two conserved quantities by \eqref{eq:fbeulermodel}--\eqref{bdry_b}. The first of which is the conservation of the $L^p$ norm of vorticity:
\begin{prop}
\label{prop:vortconserv}
For any $1\leq p \leq \infty$, we have $\|\omega(t,\cdot)\|_{L^p(\calD_t)} = \|\omega_0\|_{L^p(\calD_0)}$ for all times $t$ within the lifespan of the solution to \eqref{eq:fbeulermodel}--\eqref{bdry_b}.
\end{prop}

The second conserved quantity which is crucial to our purpose is an energy-type quantity involving the bulk kinetic energy and the total length of the free interface:
\begin{prop}\label{prop: conserved}
Let $(u(t,\cdot), \calD_t)$ be the solution to \eqref{eq:fbeulermodel}--\eqref{bdry_b}. Consider
\begin{align}
E(t) := K(t) + \sigma L(t), 
\end{align} 
where
\[K(t) := \frac{1}{2}\int_{\calD_t} |u(t,x)|^2\,dx \quad\text{ and }\quad L(t):=\int_{\Gamma_t} dS_t
\]
are the kinetic energy of the fluid and the length of $\Gamma_t$ respectively. Then 
\begin{align}\label{conserved energy}
E(t) = E(0),
\end{align}
for all times within the lifespan of the solution.
\end{prop}
The proofs of these two results are standard, and we refer interested readers to \cite{hu2024small} for details.

\section{A Confinement Lemma Regarding the Free Boundary}\label{subsect:bdryctrl}
In this section, we show that if the initial domain is $B_2$ and the initial kinetic energy is sufficiently small, the free boundary $\Gamma_t$ will always be confined in an annulus about $\p B_2$, and the width of which has the same order as the size of the initial kinetic energy. Comparing to a similar statement appearing in \cite{hu2024small}, which treats the case of the periodic water waves, the proof of Lemma \ref{prop:bdryctrl} is more involved -- it requires some subtle geometric observations concerning smooth, compact planar domains.
\begin{lem}
\label{prop:bdryctrl}
Let $\sigma>0$. Consider the solution to the system \eqref{eq:fbeulermodel}--\eqref{bdry_b} with initial fluid domain $\calD_0 = B_2$ and initial velocity $u_0$ satisfying $K(0) \leq \delta\sigma$ with $\delta < \frac{2\pi}{27}$. Then for some universal constants $c_0,c_1$, we have
\begin{align}\label{est:bdryctrl}
\Gamma_t \subset B_{2+c_0\delta}\backslash B_{2-c_1\delta},
\end{align}
for all times lifespan of the solution. Moreover, we have the following bound for kinetic energy:
\begin{equation}
\label{Kt_ineq}
K(t) \leq K(0)
\end{equation}
\end{lem}

To prove the above lemma, we need the following quantitative version of isoperimetric inequality from \cite[Theorem 4]{Osserman}:
\begin{thm}\label{thm:isoperimetric}
    If a planar, rectifiable Jordan curve of length $L$ bounds a domain $D$ of area $A$, then the following inequality holds:
    \begin{equation}
        \label{est:isoperimetric}
        L^2 - 4\pi A \ge \pi^2 (R - \rho),
    \end{equation}
    where $R$ is the radius of the circumscribed circle of $D$ (circumradius), and $\rho$ is the radius of a largest inscribed circle of $D$ (inradius).
\end{thm}

With this isoperimetric inequality, we prove the geometric Lemma \ref{prop:bdryctrl}:
\begin{proof}[Proof of Lemma \ref{prop:bdryctrl}]
Using \eqref{conserved energy} and $K(t) \ge 0$ for all $t$, we have
\begin{equation}\label{estimate l(t)}
L(t) \le \sigma^{-1}K(t) + L(t) = \sigma^{-1}K(0) + L(0) \le \delta + 4\pi,
\end{equation}
where we used the assumption and $L(0) = 4\pi$ in the final inequality above. Moreover, in view of incompressibility, we also have
\begin{equation}\label{eq:volume}
|\calD_t| = |\calD_0| = |B_2| = 4\pi.
\end{equation}
Using \eqref{eq:volume}, \eqref{estimate l(t)}, and the isoperimetric inequality \eqref{est:isoperimetric}, we deduce that
\begin{equation}\label{est:deficiency}
R(t) - \rho(t) \le \frac{1}{\pi^2}(L(t)^2 -4\pi|\calD_t|) \le \frac{1}{\pi^2}\left((\delta + 4\pi)^2 - 16\pi^2\right) = \frac{\delta(8\pi + \delta)}{\pi^2} < \frac{9}{\pi}\delta,
\end{equation}
where we also used $\delta < 1$ in the last inequality. Here, $R(t)$ is the circumradius of $\calD_t$, and $\rho(t)$ is the inradius of $\calD_t$. 

To prove the proposition, we first make the following crucial geometric observations:
\begin{enumerate}
    \item The circumscribed circle for $\calD_t$ is $\p B_{R(t)}$: to see this, we consider the set $S := \{p \in \Gamma_t\;|\; |p| = \max_{q \in \Gamma_t} |q|\}$, where $|\cdot|$ denotes the Euclidean norm. Note that such $S$ is nonempty since $\Gamma_t$ is a closed, compact curve. Also, we note that if $p = (p_1,p_2) \in S$, then the points
    $$
    \overline{p} = (p_1,-p_2),\quad \tilde{p} = (-p_1,p_2),\quad\overline{\tilde{p}} = (-p_1,-p_2)
    $$
    are also members of $S$ since $\calD_t$ is 2-fold reflective symmetric. Therefore, we may without loss assume that $p$ is in the first quadrant. 

    Now, it is clear that $\calD_t \subset B_{|p|}$. We would like to show that $\p B_{|p|}$ actually is the circumscribed circle for $\calD_t$. To wit, we let $\p B_r(q)$ be the circumscribed circle for $\calD_t$. Then we claim that $q = (0,0)$. Indeed, if $q$ is in the first quadrant, then we must have $r \ge |q -\overline{\tilde{p}}| > |p|$, which contradicts the assumption that $\p B_r(q)$ is the circumscribed circle. Since $\calD_t$ is 2-fold reflective symmetric, the argument is similar if $q$ is in other quadrants, and the claim is proved. Moreover, it is clear that we must have $r \ge |p|$. Based on the above considerations, the circumscribed circle for $\calD_t$ is $\p B_{R(t)}$, where $R(t) = |p|$.
    \item $\Gamma_t \subset B_{R(t)}\backslash B_{2\rho(t)-R(t)}$: we immediately have $\Gamma_t \subset B_{R(t)}$ by Observation 1. If there exists $p \in \Gamma_t \cap B_{2\rho(t)-R(t)}$, then we observe that
    $$
    \max_{q \in \Gamma_t}|p - q| \le |p| + R(t) < 2\rho(t).
    $$
    However, this inequality implies that all inscribed circles of $\Gamma_t$ must have radii strictly less than $\rho(t)$, leading to a contradiction.
\end{enumerate}
Next, we give bounds for $R(t)$ and $\rho(t)$ using \eqref{est:deficiency}. We first note that $\rho(t) < 2$ due to incompressibility of $u$, as otherwise $|\calD_t| \ge \pi|\rho(t)|^2 > 4\pi = |\calD_0|$. This bound for $\rho(t)$ joining with \eqref{est:deficiency} yields:
\begin{equation}
    \label{est:Rupper}
    R(t) \le \frac{9\delta}{\pi} + \rho(t) \le 2 + \frac{9\delta}{\pi}.
\end{equation}
Moreover, since again $R(t) \ge 2$ due to incompressibility, an application of \eqref{est:deficiency} gives:
\begin{equation}
    \label{est:rholower}
    \rho(t) \ge R(t) - \frac{9\delta}{\pi} \ge 2-\frac{9\delta}{\pi}.
\end{equation}
Combining \eqref{est:Rupper} and \eqref{est:rholower}, we conclude that
\begin{equation}
    \label{est:inner}
    2\rho(t) - R(t) \ge 2 - \frac{27\delta}{\pi}.
\end{equation}
Hence, given $\delta < \frac{2\pi}{27}$, the bounds \eqref{est:Rupper}, \eqref{est:inner} together with Observation 2 above yield
$$
\Gamma_t \subset B_{2+c_0\delta}\backslash B_{2-c_1\delta},
$$
where $c_0 = \frac{9}{\pi}$, $c_1 = \frac{27}{\pi}$. This concludes the proof of \eqref{est:bdryctrl}.

To show \eqref{Kt_ineq}, we observe from \eqref{conserved energy} that
\begin{equation}\label{eq:energyaux1}
K(t) = K(0) + \sigma(L(0)-L(t)) = K(0) + \sigma(4\pi - L(t)).
\end{equation}
Using the classical isoperimetric inequality, we also have $L(t)^2 \ge 4\pi |\calD_t| = 16\pi^2$, which implies the lower bound $L(t) \ge 4\pi$. This lower bound combining with \eqref{eq:energyaux1} immediately gives \eqref{Kt_ineq}.
\end{proof}

\begin{rmk}
    We remark that one is able to show a similar statement to Lemma \ref{prop:bdryctrl} when the initial domain $\calD_0$ is a perturbation of $B_2$, in the sense that $||\Gamma_0| - 4\pi|$ is sufficiently small. It is unclear to the author whether one still can show that $\Gamma_t$ is sufficiently separated from the origin if $\calD_0$ is no longer a perturbation of a disk.
\end{rmk}

%%%%%%%%%%%%%%%%%%%%%%%%%%%%%%%%%%%%%

\section{An Approximated Biot-Savart Law for Droplet} \label{sec:bs}

This section is dedicated to proving a crucial characterization of the velocity field $u$ close to the origin. We identify a decomposition $u(t,\cdot) = U(t,\cdot) + e(t,\cdot)$, where $U$ is the main term which contributes to the large growth in $\nabla^2\omega$, and $e$ is an error term whose size and regularity are \textit{a priori} controlled. More precisely, we prove the following key proposition, which is in a similar spirit to \cite[Lemma 3.1]{kiselev2014small} and \cite[Lemma 2.1]{zlatovs2015exponential}.

\begin{prop}
\label{prop:key}
Assume that the initial data of the system \eqref{eq:fbeulermodel}--\eqref{bdry_b} satisfy the symmetry assumptions stated in Lemma~\ref{lem: symmetry}. There exists $\delta_0 >0$ sufficiently small that, for $K(0) \le \delta_0\sigma$ and for any $x \in B_{1/2}$ in the first quadrant, the following holds for all times $t$ in the lifespan of the solution:
\begin{equation}
\label{eq:key}
u_j(t,x) = (-1)^j\frac{4}{\pi}\left(\int_{Q(2x)}\frac{y_1y_2}{|y|^4}\omega(t,y)dy + B_j(t,x)\right)x_j,\quad j = 1,2,
\end{equation}
where $Q(x) := [x_1,1] \times [x_2,1]$, and  $B_1$ and $B_2$ satisfies
\begin{equation}\label{est:r1}
\begin{split}
|B_1(t,x)| \le C_0\left( \|\omega_0\|_{L^\infty(\calD_0)}\left(1 + \min\left\{\log\left(1 + \frac{x_2}{x_1}\right), x_2\frac{\|\nabla\omega(t,\cdot)\|_{L^\infty([0,2x_2]^2)}}{\|\omega_0\|_{L^\infty(\calD_0)}}\right\}\right) + \sqrt{K(0)}\right),\\
|B_2(t,x)| \le C_0\left(\|\omega_0\|_{L^\infty(\calD_0)}\left(1 + \min\left\{\log\left(1 + \frac{x_1}{x_2}\right), x_1\frac{\|\nabla\omega(t,\cdot)\|_{L^\infty([0,2x_1]^2)}}{\|\omega_0\|_{L^\infty(\calD_0)}}\right\}\right) + \sqrt{K(0)}\right),
\end{split}
\end{equation}
for some universal constant $C_0$.
\end{prop}

The proof of this Proposition relies on the following decomposition, which was first observed in \cite{hu2024small}: we consider the \textit{main term} $U(t,\cdot): B_{\sqrt2} \to \R^2$ defined by
\begin{equation}\label{def_U}
U(t,\cdot) = \nabla^\perp \Psi(t,\cdot),
\end{equation}
where $\Psi$ verifies the following Poisson equation:
\begin{equation}\label{def_U1}
\begin{cases}
    \Delta \Psi(t,\cdot) = \omega(t,\cdot) &\text{in $B_{\sqrt{2}}$},\\
    \Psi(t,\cdot) = 0 &\text{on $\p B_{\sqrt{2}}$}.
\end{cases}
\end{equation}
We also define the \textit{error term} $e(t,\cdot): B_{\sqrt{2}} \to \R^2$ by
\begin{equation}
\label{def_e}
e(t,\cdot) = u(t,\cdot)|_{B_{\sqrt{2}}} - U(t,\cdot).
\end{equation}
We remark that $e(t,\cdot)$ is indeed well-defined: in view of Lemma \ref{prop:bdryctrl}, there exists $\delta_0 > 0$ sufficiently small, such that for $K(0) \le \delta_0\sigma$, we have $B_{\sqrt{2}} \subset \calD_t$ for all times $t$ within the lifespan of the solution.

In Subsection \ref{subsect:error}, we study the error term $e(t,\cdot)$ and develop suitable estimates using the boundary control proved in Subsection \ref{subsect:bdryctrl}. In Subsection \ref{subsect:main}, we prove Proposition \ref{prop:key} by developing a pointwise characterization of $U$ and incorporating estimates of $e$ found in Subsection \ref{subsect:error}.

\subsection{Estimate of the Error Term}
\label{subsect:error}
We start with showing estimates of the error term $e$ close to the origin. In what follows, we demonstrate that, given the initial kinetic energy $K(0)$ being sufficiently small, the $C^1$ norm of $e$ is controlled by $K(0)^{1/2}$.

\begin{prop}\label{prop_e}
Let $\sigma>0$. Consider the solution to the system \eqref{eq:fbeulermodel}--\eqref{bdry_b} with initial domain $\calD_0 = B_2$ and initial velocity $u_0$ such that $K(0) \le \delta_0\sigma$. Moreover, assume that $u_0$ satisfies the assumptions of Lemma \ref{lem: symmetry}. Let $U(t,\cdot)$ and $e(t,\cdot)$ be defined as in \eqref{def_U} and \eqref{def_e} respectively. Then 

\begin{enumerate}
\item $e(t,\cdot)$ is smooth in $B_{1/2}$ up to its boundary, and there exists a universal constant $C$ such that
\begin{equation}
\label{est_e}
\|\nabla e(t,\cdot)\|_{L^\infty(B_{1/2})} \leq C \sqrt{K(0)}.
\end{equation}

\item $e_1$ is odd (even) in $x_1$ ($x_2$) and $e_2$ is even (odd) in $x_1$ ($x_2$). In particular, we have the following estimate:
\begin{align}
\label{est:tildeu}
|e_j(t,x)| \leq C\sqrt{K(0)} |x_j|,\quad j = 1,2.
\end{align}
for all $x \in B_{1/2}$.
\end{enumerate}
\end{prop}

\begin{proof}
For the rest of the proof, we omit the dependence on time variable $t$ for the sake of simplicity as time dependence will not play a role in the proof. 
\begin{enumerate}
\item To prove the first statement in the proposition, we start with noticing that $\nabla\cdot e = \nabla\cdot (u - U) = 0$ in $B_{\sqrt{2}}$. Then in view of Helmholtz decomposition, there exists a stream function $F: B_{\sqrt{2}} \to \R$ such that
$
e = \nabla^\perp F
$
in $B_{\sqrt{2}}$. Moreover, we recall from \eqref{def_U} and \eqref{def_U1} that $\nabla^\perp U = \Delta \Psi = \omega$ in $B_{\sqrt{2}}$, and therefore
$$
\nabla^\perp \cdot e = \nabla^\perp \cdot (u - U) = 0
$$
in $B_{\sqrt{2}}$. The two facts established above immediately implies that in $B_{\sqrt{2}}$:
$$
\Delta F = \nabla^\perp \cdot \nabla^\perp F = \nabla^\perp\cdot e = 0
$$

Next, we show an \textit{a priori} $L^2$ estimate for the error term $e$. We first observe that $F$ and $\Psi$ are orthogonal in $\dot{H}^1(B_{\sqrt{2}})$: 
\begin{align*}
\int_{B_{\sqrt{2}}} \nabla\Psi \cdot \nabla F dx &= -\int_{B_{\sqrt{2}}} \Psi \Delta F dx + \int_{\p B_{\sqrt{2}}} \Psi (n\cdot\nabla F) dS(x) = 0,
\end{align*}
where we used $F$ being harmonic in $B_{\sqrt{2}}$ and $\Psi = 0$ on $\p B_{\sqrt{2}}$. Here, $n$ denotes the outward unit normal along $\p B_{\sqrt{2}}$, and $dS(x)$ denotes the induced surface measure on $\p\Omega$. From the orthogonality and \eqref{Kt_ineq}, we deduce that
\begin{align*}
\int_{B_{\sqrt{2}}} |\nabla \Psi|^2 + |\nabla F|^2 dx &= \|u\|_{L^2(B_{\sqrt{2}})}^2 \le 2K(t) \le 2K(0),
\end{align*}
which implies that $\|e\|_{L^2(B_{\sqrt{2}})}^2 = \|\nabla F\|_{L^2(B_{\sqrt{2}})}^2 \le 2K(0)$. To obtain higher order control for $e$, we note that both $e_1$ and $e_2$ are harmonic in $B_{\sqrt{2}}$ because so is $F$. Using this fact, the above $L^2$ control, a Calder\'on-Zygmund estimate (see e.g. \cite[Remark 2.13]{Fern_ndez_Real_2022}), and Sobolev embeddings, we conclude that
$$
\|e\|_{C^{1}(B_{1/2})} \lesssim \|e\|_{H^3(B_{1/2})} \lesssim \|e\|_{L^2(B_{\sqrt{2}})}\le C \sqrt{K(0)},
$$
which implies \eqref{est_e}.

\item To prove the second statement, we need to use the symmetry properties of the solution. We will show the case for $e_1$, and that for $e_2$ follows from a similar argument. By Lemma~\ref{lem: symmetry}, we have that $u_1(x)$ is odd in $x_1$ and even in $x_2$. Moreover, we recall from Remark \ref{rmk:vortsym} that $\omega(x)$ is odd in both $x_1$ and $x_2$. Then $\Psi(x)$ is also odd in both $x_1$ and $x_2$ for $x \in B_{\sqrt{2}}$ thanks to the uniqueness of solutions to the Poisson equation \eqref{def_U1}. Since $U_1 = \p_2 \Psi$, we know that $U_1$ is is odd in $x_1$ and even in $x_2$ in $B_{\sqrt{2}}$. Finally using the definition $e = u - U$, we conclude that $e_1$ verifies the desired symmetry properties.

To show \eqref{est:tildeu}, we first note that $e_1(0,x_2) = 0$ for any $(0,x_2) \in B_{\sqrt{2}}$ due to odd-in-$x_1$ symmetry. Then for any $(x_1,x_2) \in B_{1/2}$:
$$
|e_1(x_1,x_2)| = |e(x_1,x_2) - e(0,x_2)| \le \|e\|_{C^1(B_{1/2})}|x_1| \le C\sqrt{K(0)}|x_1|,
$$
where we invoked \eqref{est_e} at the end. The corresponding estimate for $e_2$ follows similarly, and we omit the argument here.
\end{enumerate}
\end{proof}

\subsection{Proof of the Approximate Biot-Savart Law}\label{subsect:main}

With estimates established in the previous section, we now proceed to proving Proposition \ref{prop:key}, which provides a pointwise description of the velocity field $u$ in a region away from the free interface.

\begin{proof}[Proof of Proposition \ref{prop:key}]
Let us only consider the $j=1$ case, as the case for $j=2$ follows from a symmetric argument. Throughout the proof, we will always consider $x \in B_{1/2} \cap Q_1$, where $Q_1$ denotes the first quadrant.

We start with studying the main term $U(t,x)$. Recalling \eqref{def_U1} and using the Green's function for $B_{\sqrt{2}}$ (see e.g. \cite[Chapter 2]{evans2022partial}), we can write
$$
\Psi(t,x) = \frac{1}{2\pi}\int_{B_{\sqrt{2}}}\left(\log|x-y| - \log\left(\frac{|y|}{\sqrt{2}}|x - y^*|\right)\right)\omega(t,y)dy,
$$
where $y^* = \frac{2y}{|y|^2}$. Applying \eqref{def_U}, we arrive at
\begin{equation}
    \label{eq:Urep0}
    \begin{split}
    U(t,x) &= \frac{2x_1}{\pi}\int_{B_{\sqrt{2}}\cap Q_1} \left(\frac{y_1(x_2 - y_2)}{|x-y|^2|x-\tilde{y}|^2} - \frac{y_1(x_2 + y_2)}{|x+y|^2|x-\bar{y}|^2}\right)\omega(t,y)dy\\
    &- \frac{2x_1}{\pi}\int_{B_{\sqrt{2}}\cap Q_1}\left(\frac{y_1^*(x_2 - y_2^*)}{|x-y^*|^2|x-\tilde{y}^*|^2} - \frac{y_1^*(x_2 + y_2^*)}{|x+y^*|^2|x-\bar{y}^*|^2}\right)\omega(t,y)dy\\
    &= x_1(I_1 - I_2),
    \end{split}
\end{equation}
where we used the notation that for any $p = (p_1, p_2)$, $\bar{p} = (p_1, -p_2)$ and $\tilde{p} = (-p_1, p_2)$. Note that thanks to $B_{\sqrt{2}} \subset \calD_t$, we also used the symmetry indicated in Remark \ref{rmk:vortsym} that $\omega(t,y) = -\omega(t,\bar{y}) = -\omega(t,\tilde{y})$ to reduce the integral on $B_{\sqrt{2}}$ to that on $B_{\sqrt{2}}\cap Q_1$.

We first consider $I_2$ and show that it is a harmless error. To wit, we observe that for $y \in B_{\sqrt{2}}$, we must have $|y^*| = 2|y|^{-1} > \sqrt{2}$, from which we deduce that
$$
\min\{|x-y^*|,|x-\tilde{y}^*|,|x-\bar{y}^*|,|x+y^*|\} \ge |y^*| - |x| > \frac12
$$
as $|x| < \frac12$. Therefore, we may estimate that
$$
\left|\frac{y_1^*(x_2 - y_2^*)}{|x-y^*|^2|x-\tilde{y}^*|^2} \right| \le \left|\frac{y_1^*}{|x-y^*||x-\tilde{y}^*|^2}\right| \lesssim \frac{1}{|y|},\quad \left|\frac{y_1^*(x_2 + y_2^*)}{|x+y^*|^2|x-\bar{y}^*|^2}\right| \le \left|\frac{y_1^*}{|x+y^*||x-\bar{y}^*|^2}\right| \lesssim \frac{1}{|y|}.
$$
We may then bound $I_2$ by
\begin{equation}
    \label{est:I2}
    |I_2| \le C\int_{B_{\sqrt{2}}\cap Q_1}\frac{|\omega(t,y)|}{|y|}dy \le C\|\omega_0\|_{L^\infty(B_{\sqrt{2}})},
\end{equation}
where $C$ is a universal constant.

To treat $I_1$, we may further decompose this term as:
\begin{align*}
I_1 &= \frac{2}{\pi}\int_{[0,1]^2}\left(\frac{y_1(x_2 - y_2)}{|x-y|^2|x-\tilde{y}|^2} - \frac{y_1(x_2 + y_2)}{|x+y|^2|x-\bar{y}|^2}\right)\omega(t,y)dy\\
&+ \frac{2}{\pi}\int_{(B_{\sqrt{2}}\cap Q_1)\backslash[0,1]^2}\left(\frac{y_1(x_2 - y_2)}{|x-y|^2|x-\tilde{y}|^2} - \frac{y_1(x_2 + y_2)}{|x+y|^2|x-\bar{y}|^2}\right)\omega(t,y)dy\\
&= I_{11} + I_{12}.
\end{align*}
We first estimate $I_{12}$: since $|x| < \frac12$ and $|y| > 1$ in this case, it is straightforward to see that $\min\{|x-y|, |x+y|, |x-\bar{y}|, |x-\tilde{y}|\}\ge \frac12$. Moreover, as $|y| < \sqrt{2}$, we simply conclude that
\begin{equation}
    \label{est:I12}
    |I_{12}| \le C\int_{B_{\sqrt{2}}\cap Q_1}|\omega(t,y)|dy \le C\|\omega_0\|_{L^\infty(B_{\sqrt{2}}\cap Q_1)},
\end{equation}
where $C$ is a universal constant. As for the term $I_{11}$, we note that it is exactly the same as \cite[Equation (2.5)]{zlatovs2015exponential}. Thus, an identical argument to that appearing in \cite[Lemma 2.1]{zlatovs2015exponential} plus the estimates \eqref{est:I2}, \eqref{est:I12} yields
\begin{equation}
    \label{est:I1a}
    U_1(t,x) = -\frac{4}{\pi}\left(\int_{Q(2x)}\frac{y_1y_2}{|y|^4}\omega(t,y)dy + \tilde{B}_1(t,x)\right)x_1,
\end{equation}
where
\begin{equation}
    \label{est:I1b}
    |\tilde{B}_1(t,x)| \le C\left(\|\omega_0\|_{L^\infty(B_{\sqrt{2}})} + \min\left\{\|\omega_0\|_{L^\infty(B_{\sqrt{2}})}\log\left(1 + \frac{x_2}{x_1}\right), x_2{\|\nabla\omega(t,\cdot)\|_{L^\infty([0,2x_2]^2)}}\right\}\right).
\end{equation}
Here, $C$ is some universal constant. In view of the decomposition $u_1 = U_1 + e_1$, the proof is complete after we combine \eqref{est:I1a}, \eqref{est:I1b} together with the bound \eqref{est:tildeu}.
\end{proof}

\section{Proof of the Main Theorem}\label{sec:main}
With the key Proposition \ref{prop:key} established, we may follow the ideas from \cite{zlatovs2015exponential} to establish the desired growth result regarding the vorticity Hessian. The subtleties involved in this proof is that one needs to carefully keep track of the appearance of $\eps$, i.e. the size of the initial velocity $u_0$, so as to obtain a precise rate of growth.

\begin{proof}[Proof of Theorem~\ref{thm:smallscale}]
Fix the surface tension coefficient $\sigma = 1$. We start with constructing the initial velocity $u_0 = \eps v_0$ which leads to the desired growth result. Recalling that $\calD_0 = B_2$, we define the velocity profile $v_0 = \nabla^\perp \psi$, where $\psi$ verifies
$$
\begin{cases}
    \Delta \psi = f& \text{in $\calD_0$},\\
    \psi = 0& \text{on $\p\calD_0$}.
\end{cases}
$$
Here, we pick $f: \calD_0 \to [-1,1]$ to be a smooth function which is odd in both $x_1$ and $x_2$, nonnegative in $\calD_0 \cap Q_1$, and $f \equiv 1$ on a set in $\calD_0 \cap Q_1$ with measure $1-\eta,$ where $\eta < 1$ is a sufficiently small number which will be chosen later. Moreover, we set
$$
f(x_1,x_2) = \sin^3(x_1)\sin(x_2)
$$
for $|x_1|,|x_2| < \frac{\eta}{2}$. Observe that as $f$ is odd in both $x_1$ and $x_2$, then so is $\psi$ by definition. Therefore, the initial data $(u_0, \calD_0)$ satisfy the symmetry assumptions stated in Lemma \ref{lem: symmetry}, and such symmetry is propagated. Moreover, $\|f\|_{L^2(\calD_0)} \le C\|f\|_{L^\infty(\calD_0)} = C$ for some universal constant $C$. Then
\begin{equation}
\label{est:initenergy}
K(0) = \frac12\|u_0\|_{L^2(\calD_0)}^2 =\frac{\eps^2}{2}\|v_0\|_{L^2(\calD_0)}^2 = \frac{\eps^2}{2}\|\nabla\psi\|_{L^2(\calD_0)}^2 \le C_1\eps^2,
\end{equation}
where $C_1$ is some universal constant. Note that we used a standard elliptic estimate in the final inequality. Upon choosing $\eps < \sqrt{\delta_0C_1^{-1}}$, where $\delta_0$ is chosen as in Proposition \ref{prop:key}, the initial data $(u_0,\calD_0)$ particularly verifies the assumptions stated in Lemma \ref{prop:bdryctrl} and the key Proposition \ref{prop:key}.

Let $T_0$ be the maximal lifespan of the solution $(u,\calD_t)$ corresponding to initial data $(u_0,\calD_0)$ defined above. Moreover, we define a time $T_1 := \frac{1}{\eps}\left|\log\frac{\eta}{4}\right|$. If $T_0 \le T_1$, then we are already done. Supposing the otherwise, we fix arbitrary $T \in (T_1, T_0)$. We first track the particle trajectory of a specific point, namely:
\begin{equation}
    \label{eq:flowmap}
    \begin{cases}
        \frac{d\Phi}{dt} = u(t,\Phi(t)),\\
        \Phi(0) = (e^{-\eps T},e^{-a\eps T}),
    \end{cases}
\end{equation}
where $a > 1$ is a universal number to be determined later. Since $\omega_0 = \eps f$ and $\omega(t,x)$ is transported in $\calD_t$, it is clear that for all $t \in [0, T_0)$,
\begin{equation}
    \label{eq:particle}
    \omega(t,\Phi(t)) = \omega_0(e^{-\eps T},e^{-a\eps T}) = \eps\sin^3(e^{-\eps T})\sin(e^{-a\eps T}) \ge c_2\eps e^{-(3+a)\eps T},
\end{equation}
where $c_2 > 0$ is a universal number. Moreover, we consider the time $T' = \min(T,T_*)$, where $T_* > 0$ is the first time that $\Phi(t)$ exits the region $[0,e^{-\eps T}]^2$. With the preparations above, we discuss the following two cases:

\textbf{Case 1.} $\sup_{t \le T}\|\nabla\omega(t,\cdot)\|_{L^\infty([0,2e^{-\eps T}]^2)} > 2\eps e^{\eps T}$: in this case, there exists $t'\in [0,T]$, $x' \in [0,2e^{-\eps T}]^2$ such that $|\nabla \omega(t',x')| \ge 2\eps e^{\eps T}$. Moreover, invoking the symmetry properties stated in Lemma \ref{lem: symmetry}, we have $\nabla \omega(t,0) = \nabla\omega_0(0) = 0$ for all $t \ge 0$. We thus conclude that
\begin{equation}\label{est:case1}
\sup_{t \le T}\|\nabla^2\omega(t,\cdot)\|_{L^\infty([0,2e^{-\eps T}]^2)} \ge \frac{|\nabla\omega(t',x')-\nabla\omega(t',0)|}{|x'|} \ge \eps e^{2\eps T} \ge \eps e^{\eps T},
\end{equation}
which yields the desired result.

\textbf{Case 2.} $\sup_{t \le T}\|\nabla\omega(t,\cdot)\|_{L^\infty([0,2e^{-\eps T}]^2)} \le 2\eps e^{\eps T}$: in this case, for $t \in [0,T']$, we observe that $0 \le \Phi_2(t) \le e^{-\eps T}$, which implies
\begin{equation}
    \label{est:mainerror1}
    \Phi_2(t) \|\nabla\omega(t,\cdot)\|_{L^\infty([0,2\Phi_2(t)]^2)} \le 2\eps.
\end{equation}
Combining \eqref{est:mainerror1}, \eqref{est:initenergy} with \eqref{est:r1} (Note that we may apply Proposition \ref{prop:key} here since $\Phi(t) \in [0,e^{-\eps T}]^2 \subset B_{1/2}$ by our choice of $T$), we have
\begin{equation}\label{est:bjest}
B_j(t,\Phi(t)) \le C_0(3\eps + \sqrt{C_1}\eps) \le C_2\eps,\quad j = 1,2,
\end{equation}
for all $t \in [0,T']$, where $C_2$ is a universal constant. Moreover, using the fact that $\omega(t,x) \ge 0$ in $\calD_t \cap Q_1$, and that $|\{x \in \calD_t\cap Q_1\;|\; \omega(t,x) \neq \eps\}| = \eta$ by incompressibility, we may use a similar argument to \cite[Equation (4.2)]{hu2024small} to deduce the following lower bound:
$$
\int_{Q(2x)}\frac{y_1y_2}{|y|^4} \omega(t,y) dy \ge \frac{\pi\eps}{96}\log \eta^{-1}.
$$
By choosing $\eta$ possibly even smaller that
$$
\frac{4}{\pi}\left(\frac{\pi}{96}\log \eta^{-1} - C_2\right) > \frac{1}{48}\log \eta^{-1},
$$
we invoke Proposition \ref{prop:key} to arrive at
\begin{equation}
    \label{est:u1}
    u_1(t,\Phi(t)) \le -\frac{\Phi_1(t)}{48}\eps\log\eta^{-1},
\end{equation}
\begin{equation}
    \label{est:u2}
    u_2(t,\Phi(t)) \ge \frac{\Phi_2(t)}{48}\eps\log\eta^{-1},
\end{equation}
for all $t \in [0,T']$. With \eqref{est:u1} and \eqref{est:u2} above, we start with estimates regarding the exiting time $T_*$. First, note that $\Phi(t) < e^{-\eps T}$ from \eqref{est:u1}. This means that the trajectory $\Phi(t)$ cannot leave $[0,e^{-\eps T}]^2$ across the right edge. Due to odd symmetry, it can neither leave the region across the left nor the bottom edge. Hence, we must have $\Phi_2(T_*) = e^{-\eps T}$. Moreover by \eqref{eq:flowmap} and \eqref{est:u2}, one has $\Phi_2'(t) \ge \frac{\Phi_2(t)}{48}\eps\log\eta^{-1}$ with $\Phi_2(0) = e^{-a\eps T}$, which implies the lower bound:
$$
\Phi_2(t) \ge e^{\frac{\eps}{48}\log\eta^{-1}t - a\eps T} \ge e^{-\eps T}
$$
whenever $t \ge 48 (\log\eta^{-1})^{-1}(a-1)T$. Hence by definition of $T_*$, we must have $T_* \le 48 (\log\eta^{-1})^{-1}(a-1)T$. By further choosing $\eta$ smaller that $48 (\log\eta^{-1})^{-1}(a-1) < 1$, we can ensure that $T' = T_* < T$, which also implies that $\Phi_2(T') = \Phi_2(T_*) = e^{-\eps T}$.

Now, we examine the scenario at the time instance $T'$. Thanks to the key Proposition \ref{prop:key}, the particle trajectories in $[0,e^{-\eps T}]^2$ are approximately hyperbolas, in a sense that
\begin{equation}
    \label{est:hyperbola}
    \left|\frac{d}{dt}\log(\Phi_1\Phi_2)(t)\right| \le |B_1(t,\Phi(t)) + B_2(t,\Phi(t))| \le 2C_2\eps,
\end{equation}
which gives rise to the following bound for $\Phi_1(T')$:
\begin{equation}
    \label{est:phi1refined}
    \begin{split}
    \log\Phi_1(T') &\le 2C_2\eps T' + \log \Phi_1(0) + \log \Phi_2(0) - \log \Phi_2(T')\\
    &= 2C_2 \eps T'-a\eps T \le (-a + 2C_2)\eps T,
    \end{split}
\end{equation}
This bound together with \eqref{eq:particle} and the fact that $\omega(T',0,e^{-\eps T}) = 0$ immediately implies that
\begin{equation}
    \label{est:gradientlb}
    |\p_1\omega(T',\tilde{x}_1,e^{-\eps T})| \ge c_2\eps\frac{e^{-(3+a)\eps T}}{e^{(-a+2C_2)\eps T}} = c_2\eps e^{(-3-2C_2)\eps T}
\end{equation}
for some $\tilde{x}_1 \in (0,e^{(-a + 2C_2)\eps T}]$. To proceed, we claim that $\p_1\omega(T',0,e^{-\eps T}) = 0$. In order to see this, we apply $\p_1$ to the vorticity equation $\p_t \omega + u\cdot \nabla\omega = 0$ to obtain that
\begin{equation}\label{eq:p1omega}
\p_t \p_1 \omega + \p_1u_j\p_j\omega + u_j\p_j\p_1\omega = 0.
\end{equation}
Moreover, we observe that, by symmetry of the solution, we have $\p_2u_1(t,0,x_2) = u_1(t,0,x_2) = 0$ and $\omega(t,0,x_2) = 0$, from which we also get $\p_1 u_2(t,0,x_2) = \p_2u_1(t,0,x_2) -\omega(t,0,x_2) = 0$ for all $(0,x_2) \in \calD_t$. Therefore, restricting \eqref{eq:p1omega} to $x_1 = 0$ and using the observations above, we conclude that
\begin{equation}\label{eq:p1omega2}
\p_t g + \p_1u_1(t,0,x_2) g + u_2(t,0,x_2)\p_2g = 0,
\end{equation}
where $g(t,x_2) := \p_1\omega(t,0,x_2)$. Moreover, $g(0,x_2) = 0$ for $x_2 \in (-\eta/2, \eta/2)$ by the construction of $\omega_0$. By method of characteristics, we see that
\begin{equation}\label{eq:g}
g(t,x_2) = g(0,\phi^{-1}(t,x_2))\exp\left(-\int_0^t \p_1u_1(t,0,\phi(s,\phi^{-1}(t,x_2)))ds\right),
\end{equation}
where $\phi$ solves:
$$
\frac{d\phi(t,x_2)}{dt} = u_2(t,0,\phi(t,x_2)),\quad \phi(0,x_2) = x_2.
$$
In view of Proposition \ref{prop:key}, we see that $u_2(t,0,z) > 0$ for all $0 < z < \eta/2$ and $u_2(t,0,z) < 0$ for all $0 > z > -\eta/2$. This implies that $|\phi^{-1}(t,x_2)| < |x_2|$ for all $|x_2| < \eta/2$ and $t \ge 0$, and therefore $g(0,\phi^{-1}(t,x_2)) = g(0,x_2) = 0$ for all $x_2 \in (-\eta/2, \eta/2)$. This fact plus \eqref{eq:g} yields that $g(t,x_2) = 0$ for $x_2 \in (-\eta/2, \eta/2)$. The claim is thus proved since $e^{-\eps T} < \frac{\eta}{2}$. 

With the claim $\p_1\omega(T',0,e^{-\eps T}) = 0$, we conclude by observing that
\begin{equation}\label{est:case2}
    |\nabla^2\omega(T',\tilde{x}_1,e^{-\eps T})| \ge \frac{c_2\eps e^{(-3-2C_2)\eps T}}{e^{(-a+2C_2)\eps T}} = c_2\eps e^{(a -3-4C_2)\eps T}  = \eps e^{\eps T},
\end{equation}
after we fix $a = 4 + 4C_2 - \log c_2$. Thus, \eqref{est:main} follows from \eqref{est:case1} and \eqref{est:case2}.
\end{proof}

\bibliographystyle{acm}

\end{document}